\documentclass[12pt]{article}

\usepackage[centertags]{amsmath}

\usepackage{amssymb}
\usepackage{amsthm}
\usepackage[misc]{ifsym}
\usepackage{mathrsfs}
\usepackage{graphicx}
\usepackage{epstopdf}
\usepackage{pgf,tikz}
\usetikzlibrary{arrows}
\usepackage{lineno}
\usepackage{authblk}
\usepackage{comment}

\theoremstyle{plain}
\newtheorem{thm}{Theorem}[section]
\newtheorem{cor}[thm]{Corollary}
\newtheorem{lem}[thm]{Lemma}
\newtheorem{prop}[thm]{Proposition}

\theoremstyle{definition}
\newtheorem{defn}[thm]{Definition}
\newtheorem{exam}[thm]{Example}
\newtheorem{rem}[thm]{Remark}

\theoremstyle{remark}
\numberwithin{equation}{section}

\begin{document}

\title{A matrix realization of spectral bounds}

\author[*]{Yen-Jen Cheng \Letter}
\author[*]{Chih-wen Weng}
\affil[*]{Department of Applied Mathematics, National Yang Ming Chiao Tung University, 1001 University Road, Hsinchu, Taiwan}

\maketitle

\begin{abstract}
We give a unified and systematic way to find bounds for the largest real eigenvalue of a nonnegative matrix by considering its modified quotient matrix. 
We leverage this insight to identify the unique class of matrices whose largest real eigenvalue is maximum among all $(0,1)$-matrices with a specified number of ones. This result resolves a problem that was posed independently by R. Brualdi and A. Hoffman, as well as F. Friedland, back in 1985.
\end{abstract}

{\bf keywords}:  nonnegative matrices, spectral radius, spectral bounds


\bigskip

\section{Introduction}\label{s1}
\subsection{Spectral radius of a nonnegative matrix}
All  vectors and matrices in this paper are over the field of real numbers. Let $C=(c_{ij})$ be an $n\times n$ matrix. The {\it spectral radius} of  $C$ is defined to be
$\rho(C):=\{|\lambda|~\colon~\lambda~\hbox{is an eigenvalue of }C\},$
where $|\lambda|$ is the magnitude of complex number $\lambda$. Let $\rho_r(C)$ denote the largest real eigenvalue of $C$ and $\rho_r(C)=\infty$ if $C$ has no real eigenvalues. A vector $(v_1, v_2, \ldots, v_n)$ is called {\it rooted} if $v_i\geq v_n\geq 0$ for $1\leq i\leq n-1.$
If $r_i:=\sum_{j=1}^n c_{ij}$ for $1\leq i\leq n$,  the tuple $(r_1,r_2,\ldots,r_n)$ is called the {\it row-sum vector} of $C$.
Our first main theorem is the following.
\bigskip

\noindent {\bf Theorem A.}
Let $M=(m_{ab})$ be an $\ell\times \ell$ matrix whose first $\ell-1$ columns and row-sum vector are all  rooted.
If $C=(c_{ij})$ is an $n\times n$ nonnegative matrix and there exists a partition $\Pi=(\pi_1, \pi_2, \ldots, \pi_\ell)$ of $[n]:=\{1, 2, \ldots, n\}$ such that
$$\max_{i\in \pi_a} \sum_{j\in \pi_b} c_{ij}\leq m_{ab} \quad \hbox{and} \quad \max_{i\in \pi_a} \sum_{j=1}^n c_{ij}\leq \sum_{c=1}^\ell m_{ac}$$
for $1\leq a\leq \ell$ and $1\leq b\leq \ell-1$, then $\rho(C)\leq \rho_r(M).$

\bigskip
The cases $\ell=n$ and $c_{ij}\leq m_{ij}$ for all $i,j$, and the case $\ell=1$ in Theorem A are two well-known applications of Perron-Frobenius Theorem, cf. Lemma~\ref{le2.2} and Lemma~\ref{le2.3}  in Section~\ref{s2}.
A very special situation of the case $\ell=n=2$ is
$$\rho \begin{pmatrix} c_{11} & c_{12} \\ c_{21} & c_{22}\end{pmatrix} \leq \rho_r\begin{pmatrix} 5 & 2 \\ 4 & -1 \end{pmatrix}=2+\sqrt{17}$$
for nonnegative numbers $c_{11}$, $c_{12}$, $c_{21}$, $c_{22}$  satisfying $c_{11}\leq 5$,~$c_{21}\leq 4$, $c_{11}+c_{12}\leq 7$ and $c_{21}+c_{22}\leq 3$.
The above upper bound $2+\sqrt{17}$ is smaller than the maximum row-sum $7$ of $C$, an upper bound obtained by applying Theorem~A with $\ell=1$ and $n=2$.
Indeed using $\ell=2$ in Theorem~A to compute $\rho(M)$ for a particular $M$, a lot of existing spectral bounds of nonnegative square matrices of arbitrary orders which involve square roots in their expressions can be easily reproved, cf. the results in  \cite{bh:85, cls:13, dz:13, h:98, hsf:01, hw:14, lw:13, lw:15, s:87, sw:04}  to name a few. The matrices $C$ with $\rho(C)=\rho_r(M)$ in Theorem A are also determined. Corollary~\ref{cor5.1'} will give the detailed description.

Theorem A has a dual version that deals with lower bounds.
In addition to the above results, Lemma~\ref{le3.1}, Lemma~\ref{le3.2} are of independent interest in matrix theory.

\subsection{Nonnegative matrices with prescribed sum of entries}

In 1964, B. Schwarz \cite{s64} discussed matrices obtained by rearranging the entries of a nonnegative matrix, focusing on the matrices with maximum and minimum spectral radius, respectively. Motivated by this seminal paper, the problem of finding the maximum spectral radius of $(0, 1)$-matrices with precisely $e$ ones  was proposed by R. Brualdi and A. Hoffman in 1976 \cite[p. 438]{bfvs:76}, and ten years later they gave the following conjecture in 1985  \cite{bh:85}.
\bigskip

\noindent {\bf Conjecture B.} If $$e=\frac{c(c-1)}{2}+t, \qquad \hbox{where~}t<c,$$
the maximum spectral radius of an undirected graph with $e$ edges and without isolated vertices is attained by taking the complete graph $K_c$ with $c$ vertices and adding a new vertex which is joined to $t$ of the vertices of $K_c$.
\bigskip

Conjecture B is partially solved by F. Friedland in 1985 \cite{f:85} and R. Stanley in 1987 \cite{s:87}, and totally solved by P. Rowlinson in 1988 \cite{r:88}. For the directed graph situation, R. Brualdi and A. Hoffman  \cite{bh:85} and F. Friedland  \cite{f:85}  believed the following conjecture.
\bigskip

\noindent {\bf Conjecture C.}  Let $\mathscr{S}(n,e)$ denote the set of $n\times n$ $(0, 1)$-matrices having exactly $e=c^2+t$ ones, where $2\leq t\leq 2c$.
If $A\in \mathscr{S}(n,e)$ attains the maximum spectral radius, then there exists a permutation matrix $P$ such that $PAP^T$ or $PA^TP^T$ has the form
\begin{equation}\label{e1.1}
\begin{array}{ll}
    \begin{pmatrix}
        J_c &
            \begin{array}{c}
                J_{\bigl\lfloor\frac{t}{2}\bigr\rfloor \times 1}\\
                O_{(c-\bigl\lfloor\frac{t}{2}\bigr\rfloor)\times 1}
            \end{array} \\
        \begin{array}{cc}
            J_{1\times \bigl\lceil\frac{t}{2}\bigr\rceil} & O_{1\times
            (c-\bigl\lceil\frac{t}{2}\bigr\rceil)}
        \end{array} &0
    \end{pmatrix}\oplus O_{n-c-1},
\end{array}
\end{equation}
if $t\not=2c-3$; and has the form
\begin{equation}\label{e1.2}
    \begin{pmatrix}
        J_{c-1} & J_{(c-1)\times 2}\\
        J_{2\times (c-1)} & O_2
     \end{pmatrix}\oplus O_{n-c-1}
\end{equation}
if $t=2c-3$,
where $J_{s\times t}$ is the $s\times t$ all-one matrix, $O_{s\times t}$  is the $s\times t$ zero matrix, $J_s=J_{s\times s}$, $O_s=O_{s\times s}$ and $\oplus$ is the direct sum operation of two matrices.
\bigskip

 The extremal matrices $A$ in cases $t=0$ and $t=1$ have been found in  \cite{bh:85} when Conjecture C was proposed: $PAP^T=J_c\oplus O_{n-c}$ when $t=0$;    $PAP^T=(J_c\oplus O_{n-c})+E_{ij}$ when $t=1$, where $ij$ is a position that $(J_c\oplus O_{n-c})_{ij}=0$ and $E_{ij}$ is the $n\times n$ matrix with all zero entries except for a $1$ in the $ij$ position. The cases $t=2c$, $t=2c-3$ and other $t$ much smaller than $c$ in Conjecture C was solved by F. Friedland   \cite{f:85}.  Snellman proved Conjecture C for relatively large $t$ by using combinatorial reciprocity theorem in 2003 \cite{s:03}. We will apply Theorem A to prove Conjecture~C in Section \ref{s6.1}.

Conjecture B has a generalized version that deals with $(0,1)$-matrices with zero trace.
 This line of study is usually parallel to the study of Conjecture C.
More recent result is in 2015 \cite{jz:15}, when Y. Jin and X. Zhang consider the case that $t$ is much smaller than $c$. We solve this completely in Section \ref{s6.2}.

We hope the method developed in this paper provides an efficient way of solving more extremal problems related to graphs and their spectral radii.

\subsection{Organization of the paper}

The paper is organized as follows. In Section~\ref{s2} we give some preliminaries.
Theorem \ref{th3.4} in Section~\ref{s3} is our key tool, which is not easy to apply since rooted eigenvectors are involved. In Section~\ref{s4} we provide a method to construct matrices having rooted eigenvectors. In Section~\ref{s5},
we prove  Corollary~\ref{cor5.1'}, which is a strengthening of Theorem A.  In Section~\ref{s6} we provide the following three applications of Theorem A:
\begin{enumerate}
\item[1.] Prove Conjecture C;
\item[2.] Prove the nonsymmetric matrix version of Conjecture B;
\item[3.] Determine the matrix whose spectral radius is maximum among nonnegative matrices with the largest diagonal (resp. off diagonal) element $d$ (resp. $f$) and prescribed sum of their entries.
\end{enumerate}

\section{Preliminaries}\label{s2}

Our study is based on the well-known Perron-Frobenius Theorem. Here we review the necessary parts of the theorem.

\begin{thm}[{\cite[Theorem 2.2.1]{Brou},  \cite[Corollary 8.1.29, Theorem 8.3.2]{Horn}}]\label{th2.1}
If $C$ is a nonnegative square matrix, then the following hold.
\begin{enumerate}
\item[(i)] The spectral radius $\rho(C)$ is an eigenvalue of $C$ with a corresponding nonnegative right eigenvector and a corresponding nonnegative left eigenvector.
\item[(ii)] If there exists a column vector $v> 0$ and a nonnegative number $\lambda$ such that $Cv\leq \lambda v$, then $\rho(C)\leq\lambda$.
\item[(iii)] If there exists a column vector $v\geq 0$, $v\not=0$  and a nonnegative number $\lambda$ such that  $Cv\geq \lambda v$, then  $\rho(C)\geq\lambda$.
\end{enumerate}
Moreover, if $C$ is irreducible, then
 the eigenvalue $\rho(C)$ in (i) has multiplicity $1$ and its corresponding left eigenvector and right eigenvector can be chosen to be positive, and any nonnegative left or right eigenvector of $C$ is only corresponding to the eigenvalue $\rho(C)$.
 \qed
\end{thm}

Unless specified otherwise, by eigenvector we always mean the right eigenvector.
The following two lemmas are  well-known consequences of Theorem~\ref{th2.1}. We provide their proofs since they motivate our  proofs of results.

\begin{lem}[{\cite[Theorem 2.2.1]{Brou}}]\label{le2.2}
If $0\leq C\leq C'$ are square matrices, then $\rho(C)\leq\rho(C')$. Moreover, if $C'$ is irreducible,
then $\rho(C')=\rho(C)$  if and only if $C'=C$.
\end{lem}
\begin{proof}
 Let $v$ be  a nonnegative eigenvector of $C$ for $\rho(C)$. From the assumption,  $C'v\geq Cv=\rho(C)v$.
 By applying Theorem~\ref{th2.1}(iii) with $(C, \lambda)=(C', \rho(C))$, we have  $\rho(C')\geq \rho(C).$
   Clearly $C'=C$ implies $\rho(C')=\rho(C)$.
 If $\rho(C')=\rho(C)$ and $C'$ is irreducible, then
 $\rho(C)v'^Tv=\rho(C') v'^Tv=v'^TC'v\geq v'^TCv= \rho(C)v'^Tv,$
 where $v'^T$ is a positive left eigenvector of $C'$ for $\rho(C').$
 Hence $v'^TC'v=v'^TCv$.  As $v'^T$ is positive,
  $C'v=Cv$ and $\rho(C')v=\rho(C)v=Cv=C'v.$
Since $v$ is a nonnegative eigenvector of irreducible nonnegative  matrix $C'$,
 $v$ is positive and $C'=C.$
 \end{proof}

The matrix $C'$ in Lemma~\ref{le2.2} is a {\it matrix realization} of the upper bound $\rho(C')$ of $\rho(C)$ as stated in the title. We shall provide other matrix realizations. The next one is via a $1\times 1$ matrix.

\begin{lem}[{\cite[Theorem 8.1.22]{Horn}}]\label{le2.3}
If an $n\times n$ matrix $C=(c_{ij})$ is nonnegative with row-sum vector $(r_1,r_2,\ldots,r_n)$ satisfying $r_1\geq r_i\geq r_n$ for $1\leq i\leq n$,
then
$$r_n \leq \rho(C) \leq r_1.$$
Moreover, if $C$ is irreducible, then $\rho(C)=r_1$ (resp.  $\rho(C)=r_n$) if and only if $C$ has constant row-sum.
\end{lem}

We provide a proof of the following generalized version of Lemma~\ref{le2.3}, which is due to  Ellingham and Zha \cite{ez:00}.

\begin{lem}[\cite{ez:00}]\label{le2.4}
If an $n\times n$ matrix $C$ (not necessary to be nonnegative) with row-sum vector $(r_1,r_2,\ldots,r_n)$, where $r_1\geq r_i\geq r_n$ for $1\leq i\leq n$, has a nonnegative left eigenvector $v^T=(v_1,v_2,\ldots,v_n)$ for $\theta$, then
$$r_n \leq \theta \leq r_1.$$
Moreover, $\theta=r_1$ (resp. $\theta=r_n$) if and only if $r_i=r_1$ (resp. $r_i=r_n$) for the indices $i$ with $v_i\ne 0$.
In particular, if $v^T$ is positive, $\theta=r_1$ (resp. $\theta=r_n$) if and only if $C$ has constant row-sum.
\end{lem}
\begin{proof}
Without loss of generality, let $\sum_{i=1}^{n}v_i=1$. Then
$$\theta=\theta v^TJ_{n\times 1}=v^TCJ_{n\times 1}=\sum_{i=1}^nv_ir_i.$$
So $\theta$ is a convex combination of those $r_i$ with indices $i$ satisfying $v_i> 0$, and the result follows.
\end{proof}

Let $\Pi=\{\pi_1, \pi_2, \ldots, \pi_\ell\}$  be a partition of $[n]$ and let
 $C$ be an $n\times n$ matrix.
We define an  $\ell\times \ell$ matrix $\Pi(C):=(p_{ab})$, where
$p_{ab}$ equals the average row-sum of the submatrix $C[\pi_a|\pi_b]$ of $C$.
 In matrix notation,
\begin{equation}\label{e2.1}
\Pi(C)=(S^TS)^{-1}S^TCS,
\end{equation}
where $S=(s_{jb})$ is the $n\times \ell$ {\it characteristic matrix} of $\Pi$, i.e.,
$$s_{jb}=
\left\{
\begin{array}{ll}
1, & \hbox{if $j\in \pi_b$;} \\
0, & \hbox{otherwise} \\
\end{array}
\right.$$
for $1\leq j\leq n$ and $1\leq b\leq \ell.$
The matrix $\Pi(C)$ is referred to as  the {\it quotient matrix} of $C$ with respect to $\Pi$.
Moreover if
$$p_{ab}=\sum_{j\in \pi_b} c_{ij}\qquad (1\leq a, b\leq \ell)$$
holds for every $i\in \pi_a,$ then  $\Pi$ is called an {\it equitable partition} of $C$, and
$\Pi(C)$ is called an {\it equitable quotient matrix} of $C$.
Note that $\Pi$ is an equitable partition of $C$ if and only if
\begin{equation}\label{e2.2}S\Pi(C)=CS.\end{equation}
Similarly for a column vector $u=(u_1, u_2,\ldots, u_n)^T$, we define a vector $\Pi(v)=(p_a)$ of length $\ell$, where
$p_a$ equals the average value of entries of $v$ with indices in $\pi_a$. If $u=S\Pi(u)$ then $\Pi$ is an {\it equitable partition} of $u$,
 and $\Pi(u)$ is called an {\it equitable quotient vector} of $u$.

\begin{lem}[{\cite[Lemma 2.3.1]{Brou}}]\label{le2.5}
If an $n\times n$ matrix $C$ has an equitable partition $\Pi=\{\pi_1, \pi_2, \ldots, \pi_\ell\}$ with characteristic matrix $S$,  and $\lambda$ is an eigenvalue  of $\Pi(C)$ with eigenvector $u$,   then   $\lambda$ is an  eigenvalue of $C$ with eigenvector $Su$.
\end{lem}

The following are some useful properties of equitable quotient matrices.

\begin{prop}[{\cite[Corollary 5.2.3]{g:93}}]\label{pro2.6}
If $C$ is an $n\times n$ nonnegative matrix and $\Pi$ is an equitable partition of $C$, then $\rho(C)=\rho(\Pi(C))$. \qed
\end{prop}

\begin{rem}
In  \cite[Corollary 5.2.3]{g:93}, the author only considered symmetric $(0,1)$-matrices, but the same idea proves the general case.
\end{rem}

Let $I_n$ denote the $n\times n$ identity matrix.

\begin{prop}\label{prop2.8}
Let $\Pi$ be an equitable partition of $n\times n$ matrices $C_1$ and $C_2$, then $\Pi$ is also an equitable partition of $C_1C_2$ and
$$\Pi(C_1C_2)=\Pi(C_1)\Pi(C_2).$$
In particular, if $C_1$ is invertible and $\Pi$ is an equitable partition of $C_1^{-1}$, then $\Pi(C_1^{-1})=\Pi(C_1)^{-1}$.
\end{prop}
\begin{proof}
From (\ref{e2.2}), we have $S\Pi(C_1)=C_1S$ and $S\Pi(C_2)=C_2S$, where $S$ is the characteristic matrix of $\Pi$.
By (\ref{e2.1}),
$$\Pi(C_1)\Pi(C_2)=(S^TS)^{-1}S^TC_1S\Pi(C_2)=(S^TS)^{-1}S^T_1C_1C_2S=\Pi(C_1C_2).$$
Hence $$C_1C_2S=C_1S\Pi(C_2)=S\Pi(C_1)\Pi(C_2)=S\Pi(C_1C_2).$$
By (\ref{e2.2}) again, $\Pi$ is an equitable partition of $C_1C_2$. The second part follows from $\Pi(C_1)\Pi(C_1^{-1})=\Pi(C_1C_1^{-1})=\Pi(I_n)=I_\ell$, where $\ell=|\Pi|$.
\end{proof}

\section{Spectral bound via a same size matrix}\label{s3}

In this section we develop the key tool in this paper.

\subsection{A generalization of Lemma~~\ref{le2.2}}\label{s3.1}

We generalize Lemma~\ref{le2.2} in the sense of Lemma~\ref{le2.4} that the matrices considered are not necessarily nonnegative.
\begin{lem}\label{le3.1}
 Let $C=(c_{ij})$, $C'=(c'_{ij})$, $P$ and $Q$ be  $n\times n$ matrices.
Assume that
\begin{enumerate}
\item[(i)]    $PCQ\leq PC'Q$;
\item[(ii)]  $C'$ has an eigenvector $Qu$ for $\lambda'$, where $u$ is a nonnegative column vector  and $\lambda'\in \mathbb{R}$;
\item[(iii)] $C$ has a left eigenvector $v^TP$ for $\lambda$, where  $v^T$ is a nonnegative row vector and  $\lambda\in \mathbb{R}$; and
\item[(iv)] $v^TPQu>0.$
\end{enumerate}
 Then $\lambda\leq \lambda'$.
Moreover, $\lambda=\lambda'$
if and only if
\begin{equation}\label{e3.1}
(PC'Q)_{ij}=(PCQ)_{ij}\qquad \hbox{for~}1\leq i, j\leq n \hbox{~with~} v_i\ne 0 \hbox{~and~} u_j\ne 0.
\end{equation}
\end{lem}

\begin{proof}
Multiplying the nonnegative vector $u$ given in (ii) to the right of both terms of  (i), we have
\begin{equation}\label{e3.2}
PCQu\leq PC'Qu=\lambda'PQu.
\end{equation}
Multiplying the nonnegative row vector $v^T$ given in (iii) to the left of all terms  in (\ref{e3.2}), we have
\begin{equation}\label{e3.3}\lambda v^TPQu=v^TPCQu\leq v^TPC'Qu=\lambda' v^TPQu.\end{equation}
Now we delete the positive term $v^TPQu$ to obtain $\lambda\leq \lambda'$ and finish the proof of the first part.

 Assume that $\lambda=\lambda'$. Then the inequality in (\ref{e3.3}) is an equality.  Especially $(PCQu)_i=(PC'Qu)_i$
for any $i$ with $v_i\not=0.$ Hence $(PCQ)_{ij}=(PC'Q)_{ij}$ for  any $i,j$ with $v_i\not=0$ and  $u_j\not=0.$

Conversely, (\ref{e3.1}) implies
$v^TPCQu=v^TPC'Qu.$
Then $\lambda=\lambda'$ by (\ref{e3.3}).
\end{proof}

If $C$ is nonnegative and $P=Q=I_n$, then Lemma~\ref{le3.1} becomes Lemma~\ref{le2.2}
with an additional assumption $v^Tu>0$ which immediately holds if $C$ or $C'$ is irreducible by Theorem~\ref{th2.1}.

In the sequels, we shall call two statements  that resemble each other by switching $\leq$ and $\geq$ and corresponding variables, like
$\theta\geq r_n$ and $\theta\leq r_1$, as {\it dual statements}. Two proofs are called {\it dual proofs} if one proof is obtained from the other  by simply switching each of $\leq$ and $\geq$ to the other. The following is a dual version of Lemma~\ref{le3.1} which is proved by a dual proof.

\begin{lem}\label{le3.2}
 Let $C=(c_{ij})$, $C'=(c'_{ij})$, $P$ and $Q$ be  $n\times n$ matrices.
Assume that
\begin{enumerate}
\item[(i)]    $PCQ\geq PC'Q$;
\item[(ii)]  $C'$ has an eigenvector $Qu$ for $\lambda'$, where $u$ is a  nonnegative column vector  and $\lambda'\in \mathbb{R}$;
\item[(iii)] $C$ has a left eigenvector $v^TP$ for $\lambda$, where $v^T$   is a nonnegative row vector and  $\lambda\in \mathbb{R}$; and
\item[(iv)] $v^TPQu>0.$
\end{enumerate}
 Then $\lambda\geq \lambda'$.
Moreover, $\lambda=\lambda'$
if and only if
$$
(PC'Q)_{ij}=(PCQ)_{ij}\qquad \hbox{for~}1\leq i, j\leq n~ \hbox{~with~} v_i\ne 0 \hbox{~and~} u_j\ne 0.
$$\qed
\end{lem}

\subsection{The special case $P=I_n$ and a particular $Q$}\label{s3.2}

 The following matrix notation will be adopted in the paper. For a matrix $C=(c_{ij})$ and subsets $\alpha$, $\beta$ of row indices and column indices respectively, we use $C[\alpha|\beta]$ to denote the submatrix of $C$ with size $|\alpha|\times |\beta|$ that has entries $c_{ij}$ for $i\in \alpha$ and $j\in\beta$, and define
$C[\alpha|\beta):=C[\alpha|\overline{\beta}],$ where $\overline{\beta}$ is the complement of $\beta$ in the set of column indices.
We define $C(\alpha|\beta]$ and $C(\alpha|\beta)$ similarly.
We use  $i$ to denote the subset $[i]=\{1, 2, \ldots, i\}$ to reduce the double use of parentheses. For example, $C[i|i]:=C[[i]|[i]]$ and $C[i|i):=C[[i]|[i]).$

We shall apply Lemma~\ref{le3.1} and Lemma~\ref{le3.2} by letting $P=I_n$ and
\begin{equation}\label{e3.4}
Q=I_n+\sum_{i=1}^{n-1}E_{in}=\begin{pmatrix}
I_{n-1}& J_{(n-1) \times 1}\\
  O_{1\times (n-1)} & 1 \\
\end{pmatrix}.
\end{equation}
Hence for $n\times n$ matrix $C'=(c'_{ij}),$ the matrix $PC'Q$ in Lemma~\ref{le3.1}(i) is
\begin{equation}\label{e3.5}C'Q=\left(\begin{array}{cc}
C'[n-1|n-1] & \begin{array}{c} r'_1\\ r'_2 \\ \vdots \\ r'_{n-1}   \end{array}   \\
\begin{array}{cccc} c'_{n1} & c'_{n2} & \cdots & c'_{nn-1}  \end{array}         & r'_n
\end{array}\right),
\end{equation}
where $(r'_1, r'_2, \ldots, r'_n)$ is the row-sum vector of $C'.$
\bigskip

Recall that a column vector $v'=(v'_1,v'_2,\ldots,v'_n)^T$ is said to be {rooted} if $v'_j\geq v'_n\geq 0$  for $1\leq j\leq n-1$.

\begin{lem}\label{le3.3}
Let $u=(u_1, u_2, \ldots, u_n)^T$ be a column vector and $Q$ be as in (\ref{e3.4}).  Then the following (i)-(ii) hold.
\begin{enumerate}
\item[(i)] $Qu$ is rooted  if and only if  $u$ is nonnegative.
\item[(ii)] For $1\leq j\leq n-1$, $(Qu)_j>(Qu)_n$ if and only if $u_j>0 $.
\end{enumerate}
\end{lem}

\begin{proof}
(i)-(ii) follow from the observation that
$Qu=(u_1+u_n, u_2+u_n, \ldots, u_{n-1}+u_n, u_n)^T$.
\end{proof}

The following theorem is immediate from Lemma~\ref{le3.1} by applying $P=I$, the $Q$ in (\ref{e3.4}), $v'=Qu$ and referring to (\ref{e3.5}) and Lemma~\ref{le3.3}.

\begin{thm}\label{th3.4}
 Let $C=(c_{ij})$, $C'=(c_{ij})$ be  $n\times n$ matrices with row-sum vectors $(r_1, r_2, \ldots, r_n)$ and $(r'_1, r'_2, \ldots, r'_n)$ respectively. Assume the following (i)-(iv).
\begin{enumerate}
\item[(i)]   $C[n|n-1]\leq C'[n|n-1]$ and $(r_1, r_2, \ldots, r_n)\leq(r'_1, r'_2, \ldots, r'_n)$.
\item[(ii)]  $C'$ has a rooted eigenvector $v'=(v'_1, v'_2, \ldots, v'_n)^T$ for $\lambda'\in \mathbb{R}$;
\item[(iii)] $C$ has a nonnegative left eigenvector $v^T=(v_1, v_2, \ldots, v_n)$ for $\lambda\in \mathbb{R}$;
\item[(iv)] $v^Tv'>0.$
\end{enumerate}
 Then $\lambda\leq \lambda'$.
Moreover, $\lambda=\lambda'$
if and only if the following (a), (b) hold.
\begin{enumerate}
\item[(a)] If $v'_n\not=0$, then $r_i=r'_i$ for $1\leq i\leq n$ with $v_i\not=0$.
\item[(b)]
$c'_{ij}=c_{ij} \hbox{ for~}1\leq i\leq n,1\leq j\leq n-1 \hbox{~with~} v_i\ne 0 \hbox{~and~} v'_j> v'_n.$
\end{enumerate} \qed
\end{thm}

Note that (a), (b) in Theorem~\ref{th3.4} are from (\ref{e3.1}) in Lemma~\ref{le3.1}.
The first part of assumption (i) in Theorem~\ref{th3.4} indicates that in some sense the last column is {\it irrelevant} in the comparison of $C$ and $C'$. The following theorem is the dual version of Theorem~\ref{th3.4} which is proved by a dual proof.

\begin{thm}\label{th3.5}
 Let $C=(c_{ij})$, $C'=(c_{ij})$ be  $n\times n$ matrices with row-sum vectors $(r_1, r_2, \ldots, r_n)$ and $(r'_1, r'_2, \ldots, r'_n)$ respectively. Assume the following (i)-(iv).
\begin{enumerate}
\item[(i)]   $C[n|n-1]\geq C'[n|n-1]$ and $(r_1, r_2, \ldots, r_n)\geq(r'_1, r'_2, \ldots, r'_n)$, where $(r_1, r_2, \ldots, r_n)$ and $(r'_1, r'_2, \ldots, r'_n)$ are the row-sum vectors of $C$ and $C'$, respectively.
\item[(ii)]  $C'$ has a rooted eigenvector $v'=(v'_1, v'_2, \ldots, v'_n)^T$ for $\lambda'\in \mathbb{R}$;
\item[(iii)] $C$ has a nonnegative left eigenvector $v^T=(v_1, v_2, \ldots, v_n)$ for $\lambda\in \mathbb{R}$;
\item[(iv)] $v^Tv'>0.$
\end{enumerate}
 Then $\lambda\geq \lambda'$.
Moreover, $\lambda=\lambda'$
if and only if the following (a)-(b) hold.
\begin{enumerate}
\item[(a)] If $v'_n\not=0$, then $r_i=r'_i$ for $1\leq i\leq n$ with $v_i\not=0$.
\item[(b)]
$c'_{ij}=c_{ij} \hbox{ for~}1\leq i\leq n,1\leq j\leq n-1 \hbox{~with~} v_i\ne 0 \hbox{~and~} v'_j> v'_n.$
\end{enumerate}  \qed
\end{thm}

\begin{exam}
Consider the following three matrices
$$C'_\ell= \begin{pmatrix}
3 & 1 & 1\\ 0 & 0 & 3 \\ 0 & 1 & 2
\end{pmatrix}, \quad  C=\begin{pmatrix}
3 & 1 & 1\\ 1 & 0 & 2 \\ 1 & 1 & 1
\end{pmatrix}, \quad C'_r=\begin{pmatrix}
3 & 2 & 0\\ 1 & 2 & 0 \\ 1 & 2 & 0
\end{pmatrix}$$
with $C'_\ell[3|2]\leq C[3|2] \leq C'_r[3|2],$ and the same row-sum vector $(5, 3, 3)$.
Note that $C'_\ell$ has a rooted eigenvector $v'_\ell=(1, 0, 0)^T$ for $\lambda'_\ell=3$
and $C'_r$ has a rooted eigenvector $v'_r=(2, 1, 1)^T$ for $\lambda'_r=4$.
Since $C$ is irreducible, it has a positive left eigenvector $(v_1,v_2,v_3)$ for $\rho(C)$.
Hence assumptions (i)-(iv) in Theorem~\ref{th3.4} and Theorem~\ref{th3.5} hold, and
we conclude that $\lambda'_\ell\leq \rho(C)\leq \lambda'_r$.
Since $[3]\times [1]$ is the set of the pairs $(i,j)$ described in Theorem~\ref{th3.4}(b) and Theorem~\ref{th3.5}(b),
by simple comparison of
the first columns $C'_\ell[3|1]< C[3|1] = C'_r[3|1]$ of these three matrices,
we easily conclude that $3=\lambda'_\ell< \rho(C)= \lambda'_r=4$ by the second part of Theorem~\ref{th3.4} and that of Theorem~\ref{th3.5}.
\end{exam}

\section{Matrix with a rooted eigenvector}\label{s4}

To apply Theorem~\ref{th3.4} and Theorem \ref{th3.5}, we need to  construct $C'$ which possesses a rooted eigenvector for some $\lambda'$. The following lemma comes directly.

\begin{lem}\label{le4.1}
If an $n\times n$ matrix $C'$ has a rooted eigenvector for $\lambda'$, then $C'+dI_n$ also has
the same rooted eigenvector for $\lambda'+d,$ where $d$ is any constant. \qed
\end{lem}

The definition of a rooted column vector is generalized to a rooted matrix as follows.
\begin{defn}
An $n\times n$  matrix $C'=(c'_{ij})$ is called {\it rooted}  if there is a constant $d$ such that the first $n-1$  columns and the row-sum vector of $C'+dI_n$ are all rooted.
\end{defn}

The matrix $Q$ in (\ref{e3.4}) is invertible with
$$Q^{-1}=I_n-\sum_{i=1}^{n-1} E_{in}=
\begin{pmatrix}
I_{n-1}& -J_{(n-1)\times 1}
  \\
  O_{1\times (n-1)}
    & 1 \\
\end{pmatrix}.$$  Multiplying $Q^{-1}$ to the left of $C'Q$ in (\ref{e3.5}), we have
\begin{equation}\label{e4.1}
Q^{-1}C'Q=
\begin{pmatrix}
c'_{11}-c'_{n1}     & c'_{12}-c'_{n2} & \cdots     & c'_{1~n-1}-c'_{nn-1} & r'_1-r'_n \\
c'_{21}-c'_{n1}     & c'_{22}-c'_{n2} & \cdots     & c'_{2~n-1}-c'_{nn-1} & r'_2-r'_n \\
\vdots      & \vdots  & \vdots     & \vdots      & \vdots \\
c'_{n-1~1}-c'_{n1}     & c'_{n-1~2}-c'_{n2} & \cdots     & c'_{n-1~n-1}-c'_{nn-1} & r'_{n-1}-r'_{n} \\
c'_{n1}           & c'_{n2}       & \cdots     & c'_{nn-1}           & r'_n \\
\end{pmatrix}.
\end{equation}
 From (\ref{e4.1}), $C'$ is rooted if and only if $Q^{-1}(C'+dI_n)Q$ is nonnegative for some constant $d$. Moreover, $v'$ is an eigenvector of $C'$ for $\lambda'$ if and only if $u=Q^{-1}v'$ is an eigenvector of $Q^{-1}C'Q$ for $\lambda'$.

\begin{lem}\label{le4.3}
If $C'=(c'_{ij})$ is an $n\times n$ rooted matrix, then $\rho_r(C')$ exists
and $C'$ has a rooted eigenvector $v'$ for $\rho_r(C')$.
Moreover, for any eigenvalue $\lambda$ with a rooted eigenvector $v'=(v'_1, v'_2, \ldots, v'_n)^T$ of $C'$,  
the following (i), (ii) hold.
\begin{enumerate}
\item[(i)] If row vector $(c'_{n1}, c'_{n2}, \ldots, c'_{nn-1})$ is positive, then $v'$ is positive.
\item[(ii)] If  $v'$ is positive and  $r'_i> r'_n$ for some $1\leq i\leq n-1$, then $v'_i>v'_n$.
\end{enumerate}
\end{lem}

\begin{proof}
If $C'$ is a rooted matrix, then $Q^{-1}(C'+dI_n)Q$ is nonnegative for some constant $d$. By Theorem \ref{th2.1}, there is a nonnegative eigenvector $u$ of $Q^{-1}(C'+dI_n)Q$ for $\rho(Q^{-1}(C'+dI_n)Q)=\rho_r(C')+d$. Therefore, $C'$ has a rooted eigenvector $v'=Qu$ for $\rho_r(C')$.

(i) Suppose that $(c'_{n1}, c'_{n2}, \ldots, c'_{nn-1})$ is positive and $v'_n=0$. Then
$$\sum_{j=1}^{n-1}c'_{nj}v'_j=\sum_{j=1}^nc'_{nj}v'_j=(C'v')_n=\lambda v'_n=0.$$
Hence $v'$ is a zero vector, a contradiction. So $v'_n>0$ and $v'>0$ since $v'$ is rooted.

(ii) Note that  the row-sum vector  $(r'_1, r'_2, \ldots, r'_n)^T$  of $C'$ is rooted, and there exists a constant $d$ such that $\lambda+d>0$ and $C'+dI_n=(c''_{ij})$ satisfies $c''_{ij}\geq c''_{nj}$ for $1\leq i\leq n,1\leq j\leq n-1$. From the computation
$$
\begin{array}{lll}
&& \displaystyle(\lambda+d)v'_i=\sum_{j=1}^nc''_{ij}v'_j  \\
&=&\displaystyle\sum_{j=1}^{n}c''_{nj}v'_j+\sum_{j=1}^{n}(c''_{ij}-c''_{nj})v'_j  \\
&\geq &\displaystyle\sum_{j=1}^{n}c''_{nj}v'_j+\sum_{j=1}^{n}(c''_{ij}-c''_{nj})v'_n\\
&=& (\lambda+d)v'_n+(r'_i+d-r'_n-d)v'_n > (\lambda+d)v'_n,
\end{array}
$$
and deleting $\lambda+d,$ we have $v'_i>v'_n$.
\end{proof}

By Lemma \ref{le2.5} and Lemma \ref{le4.3}, we have the following lemma.
\begin{lem}\label{le4.4}
If $C'$ is an $n\times n$ matrix, $\Pi=\{\pi_1,\ldots,\pi_\ell\}$ is an equitable partition of $C'$ with $n\in\pi_\ell$ and $\Pi(C')$ is a rooted matrix, then $C'$ has a rooted eigenvector $Su$ for $\rho_r(\Pi(C'))$, where $S$ is the characteristc matrix of $\Pi$ and $u$ is a rooted eigenvector of $\Pi(C')$ for $\rho_r(C')$. \qed
\end{lem}

Lemma \ref{le4.3} and Lemma \ref{le4.4} will be useful to construct matrix $C'$ with a rooted eigenvector, while the following lemma will help us reduce the size of $C'$ for computing $\rho_r(C')$.

\begin{lem}\label{le4.5}
Let $C'$ be an $n\times n$ rooted matrix and $\Pi=\{\pi_1,\ldots,\pi_\ell\}$ be an equitable partition of $C'^T$ with $\pi_\ell=\{n\}$. Then $\rho_r(C')=\rho_r(\Pi(C'^{T}))$.
\end{lem}
\begin{proof} Let $Q$ be the same as in \eqref{e3.4}. Then $\Pi$ is an equitable partition of both $Q^T$ and $(Q^{-1})^T$. Choose a constant $d$ such that $(Q^{-1}(C'+dI_n)Q)^T$ is nonnegative. By Proposition \ref{pro2.6} and  Proposition \ref{prop2.8}, we have
$$\rho((Q^{-1}(C'+dI_n)Q)^T)=\rho(\Pi((Q^{-1}(C'+dI_n)Q)^T))=\rho(\Pi((C'+dI_n)^T)).$$
Hence $\rho_r(C')=\rho_r(\Pi(C'^{T}))$ since $\Pi((C'+dI_n)^T)=\Pi(C'^T)+dI_\ell$.
\end{proof}

The following matrix is a special rooted matrix which will be used to obtain certain bounds of the spectral radius of a nonnegative matrix.

For $d,f_1,f_2,r_1,r_2,\ldots,r_n\geq 0$ with $f_1\geq f_2$ and $r_j\geq r_n$ with $1\leq j\leq n-1$, define
\begin{equation}\label{e4.2}
M_n:=
\begin{pmatrix}
f_1J_{n-1}+(d-f_1)I_{n-1} & \left.
                              \begin{array}{c}
                                r_1-d-(n-2)f_1\\
                                r_2-d-(n-2)f_1 \\
                                \vdots\\
                                r_{n-1}-d-(n-2)f_1
                              \end{array}
                            \right. \\
f_2J_{1\times (n-1)} & r_n-(n-1)f_2
\end{pmatrix}.
\end{equation}

\begin{lem}\label{le4.6} Referring to the notation of $M_n$ in (\ref{e4.2}), the following (i), (ii) hold.
\begin{enumerate}
\item[(i)] The matrix $M_n$ has a rooted eigenvector $v'$ for the largest real eigenvalue $\rho_r(M_n)$ of $M_n$.
\item[(ii)] If $f_2>0$, then $v'>0$.
\end{enumerate}
\end{lem}
\begin{proof} Since $M_n$ is rooted, (i) follows from the first part of Lemma~\ref{le4.3}, and (ii) follows from Lemma~\ref{le4.3}(i).
\end{proof}

\begin{lem}\label{le4.7} Referring to the notation of $M_n$ in (\ref{e4.2}), the following (i)-(iii) hold.
\begin{enumerate}
\item [(i)] The largest real eigenvalue $\rho_r(M_n)$ of $M_n$ is
\begin{align}\label{e4.3}
             &\frac{1}{2}(r_n+d-f_2+(n-2)(f_1-f_2)) \nonumber \\
             &+\frac{1}{2}\sqrt{(r_n-d+f_2-(n-2)(f_1-f_2))^2+4f_2\sum\limits_{i=1}^{n-1}(r_i-r_n)}.
\end{align}
In particular   $\rho_r(M_n)\geq\max(d-f_2, r_n)$.
\item[(ii)] If $f_1=f_2=f$ and $r_n=0$, then $$\rho_r(M_n)=\frac{d-f+\sqrt{(d-f)^2+4f m}}{2},$$
where $m:=\sum_{i=1}^{n-1} r_i$ is the sum of all entries of $M_n$.
\item [(iii)]  If $f_1=f_2$ and $r_t=r_{t+1}=\cdots=r_n$ for some $t\leq n$, then $\rho_r(M_t)=\rho_r(M_n).$
\end{enumerate}
\end{lem}
\begin{proof}
Note that  $\Pi=\{\{1, 2, \ldots, n-1\}, \{n\} \}$ is an equitable partition of $M^T_n$, and
$$\Pi(M^T_n)=\begin{pmatrix} d+(n-2)f_1 & f_2 \\  \sum\limits_{i=1}^{n-1} (r_i-(d+(n-2)f_1))  & r_n-(n-1)f_2) \end{pmatrix}.$$
By Lemma \ref{le4.5} and direct computation, we find  $\rho_r(M)=\rho_r(\Pi(M^T_n))$ as the expression in  (\ref{e4.3}).
 The rest of (i),  (ii) and (iii) follow from (\ref{e4.3}) immediately.
\end{proof}

\section{The proof of Theorem A}\label{s5}

The following theorem is a strengthening of Theorem A.

\begin{thm}\label{th5.1}
Let $M=(m_{ab})$ be an $\ell\times \ell$ rooted matrix.  If $C=(c_{ij})$ is an $n\times n$ nonnegative matrix and there exists a partition $\Pi=(\pi_1, \pi_2, \ldots, \pi_\ell)$ of $[n]$ such that
$$\max_{i\in \pi_a} \sum_{j\in \pi_b} c_{ij}\leq m_{ab} \quad \hbox{and} \quad \max_{i\in \pi_a} \sum_{j=1}^n c_{ij}\leq \sum_{c=1}^\ell m_{ac}$$
for $1\leq a\leq \ell$ and $1\leq b\leq \ell-1$, then $\rho(C)\leq \rho_r(M).$
Moreover, let $u=(u_1,\ldots,u_\ell)$ be a rooted eigenvector of $M$ for $\rho_r(M)$. If $C$ is irreducible, 
then $\rho(C)= \rho_r(M)$ if and only if the following (a), (b) hold.
\begin{enumerate}
\item[(a)] If $u_\ell\not=0$, then $\sum_{j=1}^n c_{ij}= \sum_{c=1}^\ell m_{ac}$ for $1\leq a\leq \ell, i\in \pi_a$.
\item[(b)]
$\sum_{j\in \pi_b}c_{ij}=m_{ab} \hbox{ for~}1\leq a\leq \ell,1\leq b\leq \ell-1,i\in\pi_a \hbox{~with~} u_b> u_\ell.$
\end{enumerate}
\end{thm}

\begin{proof}
Rearranging the indeces of $C$ if necessary, we might assume $n \in \pi_\ell$.
We first consider the case that the row vector
$(m_{\ell 1}, m_{\ell 2}, \ldots, m_{\ell \ell-1})$  is positive.
 We construct an $n\times n$ matrix $C'$ such that the conditions in Theorem~\ref{th3.4}(i) hold, and
$\Pi$ is an equitable partition of $C'$ with $\Pi(C')=M$.
Indeed, $C'$ is obtained from $C$ by increasing entries in $C[n|n-1]$ such that each row of the submatrix $C'[\pi_a|\pi_b]$ of $C'$ has the same row-sum $m_{ab}$ for $1\leq a\leq \ell,1\leq b\leq \ell-1$, and the last column of $C'$ is filled  to make the $i$-th row-sum of $C'$ equals $\sum_{c=1}^\ell m_{ac}$ for each $i$, where $i\in\pi_a$.
By Lemma~\ref{le4.3} and Lemma~\ref{le4.4}, $\rho_r(M)$ is an eigenvalue of $C'$ with a positive and rooted eigenvector $v'$.
Let $v^T$ be a nonnegative left eigenvector of $C$ for $\rho(C)$. Then $v^Tv'>0$ and the conditions (i)-(iv) in Theorem~\ref{th3.4} hold.
From the conclusion of Theorem~\ref{th3.4} with $\lambda=\rho(C)$ and $\lambda'=\rho_r(M)$, we find  $\rho(C)\leq \rho_r(M)$.

In general,  let $\epsilon>0$ and $M_\epsilon:=M+\epsilon J_\ell$. Note that $M_\epsilon$ is rooted and $M[\{\ell\}|\{\ell\})$ is positive. Then by the argument above, we have $\rho(C)\leq \rho_r(M_\epsilon)$.
Hence  $$\rho(C)\leq \lim\limits_{\epsilon\rightarrow 0^+} \rho_r(M_\epsilon)= \rho_r(M)$$
by the continuity of the eigenvalues  \cite{e:82,np:94}.

For the second part, assume $C$ is irreducible. Then $C$ has a positive left eigenvector $v^T$ for $\rho(C)$. 
To prove the sufficiency  in the second part, assume that (a), (b) hold. Let $S$ be the charateristic matrix of $\pi$. Then by Lemma \ref{le4.4} $v'=(v'_1,\ldots,v'_n)^T=Su$ is a rooted eigenvector of above $C'$ for $\rho_r(M)$ and clearly Theorem \ref{th3.4} (a) holds. 
for $1\leq i\leq n,1\leq j\leq n-1$ with $v'_j>v'_n$. Let $i\in\pi_a$ and $j\in\pi_b$. Then $u_b=v'_j>v'_n=u_\ell$ and $\sum_{j'\in\pi_b}c_{ij'}=m_{ab}=\sum_{j'\in\pi_b}c'_{ij'}$ by (b) here. Hence for $j'\in\pi_b$, $c_{ij'}=c'_{ij'}$  since $c_{ij'}\leq c'_{ij'}$, and $c_{ij}=c'_{ij}$. 
Thus Theorem~\ref{th3.4} (b) holds and $\rho(C)=\rho_r(M)$ by Theorem~\ref{th3.4}.

To prove the necessity, assume that $\rho(C)=\rho_r(M)$.
Then clearly Theorem 3.4 (a) implies (a) here since $v'_n=u_\ell$.  For $1\leq a\leq \ell,1\leq b\leq \ell-1,i\in\pi_a \hbox{~with~} u_b> u_\ell$, from Theorem 3.4 (b), we have $\sum_{j\in\pi_b}c_{ij}=\sum_{j\in\pi_b}c'_{ij}=m_{ab}$. Then (b) holds and the proof is completed.
\end{proof}

To charaterise when the equality holds in Theorem \ref{th5.1}, we need certain information about the eigenvector of $M$. 
With additional conditions, we now give an eigenvector-free version of Theorem \ref{th5.1}. A vector $(v_1, v_2, \ldots, v_\ell)$ is {\it strictly rooted} if $v_i>v_\ell>0$ for every $i\leq \ell$. 
\begin{cor}\label{cor5.1'}
Let $M=(m_{ab})$ be an $\ell\times \ell$ rooted matrix.  If $C=(c_{ij})$ is an $n\times n$ nonnegative matrix and there exists a partition $\Pi=(\pi_1, \pi_2, \ldots, \pi_\ell)$ of $[n]$ such that
$$\max_{i\in \pi_a} \sum_{j\in \pi_b} c_{ij}\leq m_{ab} \quad \hbox{and} \quad \max_{i\in \pi_a} \sum_{j=1}^n c_{ij}\leq \sum_{c=1}^\ell m_{ac}$$
for $1\leq a\leq \ell$ and $1\leq b\leq \ell-1$, then $\rho(C)\leq \rho_r(M).$
Moreover, if $C$ is irreducible, the row vector
$(m_{\ell 1}, m_{\ell 2}, \ldots, m_{\ell\ell-1})$  is positive, and the row-sum vector of $M$ is strictly rooted,
 then $\rho(C)= \rho_r(M)$ if and only if $\Pi$ is an equitable partition of $C$ and $\Pi(C)=M$.
\end{cor}
\begin{proof}
It remians to prove the second part. 
To prove the sufficiency, assume that $\Pi$ is an equitable partition of $C$ and $\Pi(C)=M$.
Then $\rho(C)=\rho(M)= \rho_r(M)$ by Proposition~\ref{pro2.6}.

To prove the necessity, assume that $C$ is irreducible, the row vector $(m_{\ell 1}, m_{\ell 2}, \ldots, m_{\ell \ell-1})$  is positive, the row-sum vector of $M$ is strictly rooted, and $\rho(C)=\rho_r(M)$.
Then $M$ has a strictly rooted eigenvector $u=(u_1,$ $u_2,$ $\ldots,$ $u_\ell)^T$ for $\rho_r(M)$ by Lemma~\ref{le4.3}.
By the second part of Theorem \ref{th5.1}, we have that the row-sum vectors of $C$ and $C'$ are equal and $C[n|\pi_\ell)=C'[n|\pi_\ell)$. Hence $\Pi(C)=\Pi(C')=M$.
\end{proof}


We provide an example to explain Theorem~\ref{th5.1} and its proof.

\begin{exam}
Consider the following matrices $C$, $C'$ and $M$ from left to right appearing in the assumption and proof of Theorem \ref{th5.1}:
$$\left(
\begin{tabular}{ccc|cc|cc} 2 & 1 & 3 & 3 & 3 & 12 & 0\\ 4 & 2 & 1 & 4 & 2 & 6 & 4\\ 2 & 3 & 1 & 4 & 1 & 8 &3 \\
\hline
3 & 5 & 3 & 1 & 1& 3 & 4\\ 5 & 6 & 1 & 1 & 0 & 3 &3\\
\hline
0 & 2 & 1 & 2 & 2 & 6 & 0\\  2 & 2 & 0 & 2 &1& 1 &4\end{tabular}\right),
  \left(\begin{tabular}{ccc|cc|cc} 2 & 2 & 3 & 3 & 3 & 12 & -1\\ 4 & 2 & 1 & 4 & 2 & 6 & 5\\ 2 & 3 & 2 & 4 & 2 & 8 & 3 \\ \hline
 4& 5 & 3 & 1 & 1& 3 & 3\\
 5 & 6 & 1 & 1 & 1 & 3 &3\\
 \hline 1 & 2 & 1 & 2 & 2 & 6 & -1\\  2 & 2 & 0 & 2 &2& 1 &4\end{tabular}\right), \begin{pmatrix} 7 & 6 & 11 \\ 12 & 2 & 6\\ 4& 4 & 5\end{pmatrix},$$
with corresponding row-sum vectors
$$(24,23,22|20,19|13,12),\quad(24,24,24|20,20|13,13),\quad (24,20,13),$$
where the separating lines are according to the equitable partition  $\Pi=\{\{1, 2, 3\}, \{4, 5\}, \{6, 7\}\}$ of the middle matrix $C'$.
Notice that $C'$ is not rooted, so we can not apply Lemma \ref{le4.3} and then Theorem~\ref{th3.4} directly.
Since $M$ is rooted and by Theorem~\ref{th5.1}, we have
$\rho(C)\leq \rho_r(M)\approx 18.6936.$
If we apply Lemma~\ref{le2.2} with $C'$ the following nonnegative matrix $C^*$ using the same equitable partition $\Pi$:
$$C^*=\left(
\begin{tabular}{ccc|cc|cc} 2 & 2 & 3 & 3 & 3 & 12 & 0\\ 4 & 2 & 1 & 4 & 2 & 6 & 6\\ 2 & 3 & 2 & 4 & 2 & 8 &4 \\
\hline
4 & 5 & 3 & 1 & 1& 3 & 4\\ 5 & 6 & 1 & 1 & 1 & 3 &4\\
\hline
 1& 2 & 1 & 2 & 2 & 6 & 0\\  2 & 2 & 0 & 2 &2& 2 &4\end{tabular}\right), \qquad \Pi(C^*)=\begin{pmatrix} 7 & 6 & 12 \\ 12 & 2& 7\\ 4& 4 & 6\end{pmatrix},$$
then the upper bound
$\rho(C^*)=\rho (\Pi(C^*))\approx 19.4$ of $\rho(C)$  is larger than the previous one.
\end{exam}

Theorem~\ref{th5.4}, due to X. Duan and B. Zhou  \cite{dz:13}, generalizes the results in  \cite{bh:85, cls:13, h:98, hsf:01, lw:13, s:87, sw:04} and
relates to the  results in  \cite{hw:14, lw:15, s:87}. Theorem~\ref{th5.4} can be easily reproved  by using Corollary \ref{cor5.1'}.

\begin{thm}[\cite{dz:13}]\label{th5.4}
Let $C=(c_{ij})$ be a nonnegative $n\times n$ matrix with  row-sums $r_1\geq r_2\geq \cdots \geq r_n$,  $f:=\max_{1\leq i\not=j\leq n} c_{ij}$ and $d:=\max_{1\leq i\leq n} c_{ii}$. Then
\begin{equation}\label{e5.1}\rho(C) \leq \frac{r_{\ell}+d-f+\sqrt{(r_{\ell}-d+f)^{2}+4f\sum_{i=1}^{\ell-1}(r_i-r_\ell)}}{2}\end{equation}
for $1 \leq \ell \leq n.$
Moreover, if $C$ is irreducible, then the equality holds in (\ref{e5.1}) if and only if $r_1=r_n$ or for  $1\leq t\leq \ell$ with $r_{t-1}\not=r_t=r_\ell$, we have $r_t=r_n$ and
$$c_{ij}=\left\{
           \begin{array}{ll}
             d, & \hbox{if $i=j\leq t-1$;} \\
             f, & \hbox{if $i\not=j$ and $1\leq i\leq n$, $1\leq j\leq t-1.$}
           \end{array}
         \right.$$
\end{thm}
\begin{proof}
The inequality in (\ref{e5.1}) is obtained  by applying Corollary~\ref{cor5.1'} with $\Pi=\{\{1\}, \{2\}, \ldots, \{\ell-1\}, \{\ell, \ell+1, \ldots, n\}\}$ and
$M=\Pi(C'),$ where $C'=M_\ell$ is the matrix described in (\ref{e4.2}) with $n=\ell$, $f_1=f_2=f$.
For the equality case, we apply Lemma~\ref{le4.7}(iii) and the second part of Corollary~\ref{cor5.1'} by choosing the least $t$ such that $r_t=r_\ell$.
\end{proof}

\begin{rem}
From our method in Corollary~\ref{cor5.1'}, the assumptions  $f$ and $d$ in Theorem~\ref{th5.4} can be replaced by smaller numbers $f=\max_{1\leq i\leq n,1\leq j\leq \ell-1,i\ne j} c_{ij}$ and $d=\max_{1\leq i\leq \ell-1} c_{ii}$, respectively. Furthermore, by Lemma \ref{le4.7}(i) the upper bound could be replaced by
\begin{align*}&\frac{1}{2}(r_n+d-f_2+(n-2)(f_1-f_2))\\ &+\frac{1}{2}\sqrt{(r_n-d+f_2-(n-2)(f_1-f_2))^2+4f_2\sum\limits_{i=1}^{n-1}(r_i-r_n)},
\end{align*}
where $$f_1=\max_{i\in [n], j\in [\ell-1],i\ne j} c_{ij},~ f_2=\max_{\ell\leq i\leq n,j\in [\ell-1]} c_{ij},~d=\max_{i\in [\ell-1]} c_{ii}.$$
\end{rem}

The following is the dual theorem of Corollary~\ref{cor5.1'}, but its proof is not completely dual.

\begin{thm}
Let $M=(m_{ab})$ be an $\ell\times \ell$ rooted matrix. If $C=(c_{ij})$ is an $n\times n$ nonnegative matrix and there exists a partition $\Pi=(\pi_1, \pi_2, \ldots, \pi_\ell)$ of $[n]$ such that
$$\max_{i\in \pi_a} \sum_{j\in \pi_b} c_{ij}\geq m_{ab}   \quad \hbox{and} \quad \max_{i\in \pi_a} \sum_{j=1}^n c_{ij}\geq \sum_{c=1}^\ell m_{ac}$$
for $1\leq a\leq \ell$ and $1\leq b\leq \ell-1$, then $\rho(C)\geq \rho_r(M).$
Moreover, if $C$ is irreducible, the row vector
$(m_{\ell 1}, m_{\ell 2}, \ldots, m_{\ell \ell-1})$  is positive, and the row-sum vector of $M$ is strictly rooted,
 then $\rho(C)= \rho_r(M)$ if and only if $\Pi$ is an equitable partition of $C$ and $\Pi(C)=M$.
\end{thm}

\begin{proof}
The second part of the statement follows from a dual proof of  Corollary~\ref{cor5.1'}.
To prove the first part, without assuming the row vector
$(m_{\ell 1},$ $m_{\ell 2},$ $\ldots,$ $m_{\ell \ell-1})$ being positive and
referring to (\ref{e3.5}), we construct $C'$ similarly as in the proof of Theorem \ref{th5.1} by changing the operation of increasing entries to decreasing entries in $C[n|n-1]$ to have $CQ\geq C'Q\geq 0$ and $\Pi(C')=M$. Since $M$ is rooted, there is a  rooted eigenvector $v'$   of $C'$ for $\lambda'=\rho(M)$ by Lemma~\ref{le4.4}.
Then $u=Q^{-1}v'$ is nonnegative and we have $Cv'=CQu\geq C'Qu= C'v'=\lambda'v'.$
Since $v'$ is nonnegative,  $\rho(C)\geq \rho(M)$ by Theorem~\ref{th2.1}(iii).
\end{proof}

\section{Applications}\label{s6}

We provide three applications of our matrix realization of spectral bounds in this section.

\subsection{The proof of Conjecture~C}\label{s6.1}

Throughout this subsection, we assume $e=c^2+t,$ where $2\leq t\leq 2c.$
Recall that $\mathscr{S}(n,e)$ is the set of $n\times n$ $(0, 1)$-matrices having exactly $e$ $1$'s.
Let $\mathscr{S}^*(n,e)$ denote the subset of $\mathscr{S}(n,e)$ which collects matrices $A=(a_{ij})$ satisfying
$$\text{if $a_{ij}=1$, then $a_{hk}=1$ for $h\leq i$, $k\leq j$}.$$
We need the following lemma.

\begin{lem}[\cite{f:85}]\label{le6.1}
If $A\in \mathscr{S}(n,e)$ attains the maximum spectral radius, then there exists a permutation matrix $P$ such that $PAP^T\in \mathscr{S}^*(n,e)$ and $PAP^T$ has the form
$A[k|k]\oplus O_{n-k}$ for some $k$ and the submatrix $A[k|k]$ is irreducible.
\end{lem}
\bigskip

\noindent{\it Proof of Conjecture C.}

Let $A_0$ be the matrix in (\ref{e1.1}), and $A'_0$ be the matrix in (\ref{e1.2}) in the case $t=2c-3$. Observe that $A_0, A'_0\in \mathscr{S}^*(n,e)$.
To prove Conjecture C,  by using Lemma~\ref{le6.1}, we only need to  show $\rho(A)<\rho(A_0)$  for every $A\in \mathscr{S}^*(n,e)-\{A_0, A_0^T\}$ if $t\not=2c-3$; and $\rho(A)<\rho(A_0)$ for every $A\in \mathscr{S}^*(n,e)-\{A_0, A_0^T, A'_0\}$ and $\rho(A_0)<\rho(A'_0)$ if $t=2c-3$.
Fix a matrix $A\in \mathscr{S}^*(n,e)-\{A_0,A_0^T\}$. By considering $A^T$ if necessary, we might assume that
the number of $1$'s in $A[c|c)$ is no larger than that in $A(c|c]$.
Let $(r_1, r_2, \ldots, r_n)$ denote the row-sum vector of $A$. Since $e\geq c^2+2$, $A[c+1~|~c+1]$ is irreducible by Lemma~\ref{le6.1}, so $0<r_{c+1}<\max (c+1, t)$.  Let
$s:=r_{c+1}$. Applying Corollary~\ref{cor5.1'} with $\ell=c+1,$ $C=A$, partition $\Pi=\{\{1\}$, $\{2\}$, $\ldots,$ $\{c\}$,  $\{c+1, c+2,\ldots, n\}\},$ and the following $(c+1)\times (c+1)$ rooted matrix
\begin{equation*}
M=\begin{pmatrix}
J_s & J_{s\times (c-s)}                 &\begin{array}{c} r_1-c\\ r_2-c \\ \vdots   \end{array} \\
J_{(c-s)\times s}      & J_{c-s}  &\begin{array}{c}     \\  r_c-c\end{array}\\
J_{1\times s} &  O_{1\times (c-s)}               & 0
\end{pmatrix},
\end{equation*} we have $\rho(A)\leq \rho_r(M)$.
 To find $\rho_r(M)$, we consider the partition
$\Pi_1=\{\{1,2,\ldots,s\}$, $\{s+1, \ldots, c\}$, $\{c+1\}\}$ of $[c+1]$, and  observe that  $\Pi_1$ is an equitable partition of $M^T$. According to $s=c$ or $s<c$, the equitable quotient matrix $\Pi_1(M^T)$ has one of the following two forms
\begin{equation}\label{e6.1}
\begin{pmatrix}
c& 1\\
a & 0
\end{pmatrix}, \quad
\begin{pmatrix}
s & c-s   & 1 \\
s   & c-s & 0 \\
a   &  b    & 0
\end{pmatrix},
\end{equation}
respectively, where $a=\sum_{i=1}^s (r_i-c), \quad b=\sum_{i=s+1}^c (r_i-c).$
Observe that the first matrix is the {\it degenerated} case of the second: the $-1$ value and the two eigenvalues of the first matrix form the three eigenvalues of the second matrix in the special case $s=c$ and $b=0$.

Note that $A=A_0$ implies $s=\lceil\frac{t}{2}\rceil$, $a=\lfloor\frac{t}{2}\rfloor$ and $b=0$. The converse is also true: Notice that $$a+b+s+\sum_{i=c+2}^n r_i=t$$ holds generally in $A$.
The conditions $s=\lceil\frac{t}{2}\rceil$, $a=\lfloor\frac{t}{2}\rfloor$ and $b=0$ imply
$A(c+1|c]=O_{(n-c-1)\times c}$ and $A[\{s+1, s+2, \ldots, c\}|c)=O_{(c-s)\times (n-c)}$, and $A(c+1|c]=O_{(n-c-1)\times c}$ also implies
$A[c|c+1)=O_{c\times (n-c-1)}$ from the shape of $A$ described in Lemma~\ref{le6.1}, so $A=A_0$ follows.

We will provide important constraints between $s,a, b$ and $t$. Let $r$ denote the number of zeros in $A[c|c].$ Then $a+b+r$ is  the number of $1$'s in $A[c|c)$. From the assumption in the beginning, we have  $a+b+r \leq (e-(c^2-r))/2=(t+r)/2$.
Since the integer $a$ is at most the number of $1$'s in $A[s|c)$, we have
\begin{equation}\label{e6.2}
a\leq a+b+r\leq (t+r)/2.
\end{equation}
 In particular, $a+b\leq (t-r)/2$ and
\begin{equation}\label{e6.3}
2a+b\leq t.
\end{equation}
Since the number of $1$'s in $A(c|s]$ is $e-(c^2-r)-(a+b+r)=t-a-b$, we have
\begin{equation}\label{e6.4}
 s\leq t-a-b.
\end{equation}
Since $\rho_r(M)=\rho_r(\Pi_1(M^T))$ by Lemma~\ref{le4.5}, it suffices to show $\rho(\Pi_1(M^T))< \rho(A_0)$. The characteristic polynomial of $\Pi_1(M^T)$ in (\ref{e6.1}) is $f(x)/(x+1)$ or $f(x)$
according to $s=c$ or $s<c$, where
\begin{equation}\label{e6.45}
f(x)=x^3-cx^2-ax+a(c-s)-sb.
\end{equation}
We use Calculus to study the shape of the polynomial $f(x)$. The derivative of $f(x)$ is
$f'(x)=3x^2-2cx-a.$ If $x>c$,  then
$$f'(x)> c(3c-2c)-a=c^2-a\geq c^2-\frac{t+r}{2}\geq c^2-\frac{2c+(c-1)^2}{2}\geq 0$$
by (\ref{e6.2}).
Hence $f(x)$ is increasing in the interval $(c, \infty).$  Since $\rho(A_0)>\rho(J_c)=c$, to prove $\rho_r(M)<\rho(A_0)$, it suffices to show that $f(\rho(A_0))>0$.
Setting $s=a=\left\lceil\frac{t}{2}\right\rceil$ and $b=0$ in (\ref{e6.45}),
$\rho(A_0)$ is the largest zero of
\begin{equation}\label{e6.5}
g(x):=x^3-cx^2-\left\lfloor\frac{t}{2}\right\rfloor x+\left\lfloor\frac{t}{2}\right\rfloor\left(c-\left\lceil\frac{t}{2}\right\rceil\right).
\end{equation}
Then
\begin{align}\label{e6.6}
f(\rho(A_0))=&f(\rho(A_0))-g(\rho(A_0)) \nonumber \\
            =&\left(\left\lfloor \frac{t}{2}\right\rfloor-a\right) (\rho(A_0)-c) -s\left(a+b\right)+ \left\lfloor\frac{t}{2}\right\rfloor \left\lceil\frac{t}{2}\right\rceil.
\end{align}

If $a\leq \left\lfloor \frac{t}{2}\right\rfloor$, then  we immediately have $f(\rho(A_0))\geq 0$ from (\ref{e6.6}), since $s+a+b\leq t$ in (\ref{e6.4}) implies $s(a+b)\leq \left\lfloor\frac{t}{2}\right\rfloor \left\lceil\frac{t}{2}\right\rceil$; and indeed $f(\rho(A_0))> 0$, since  $f(\rho(A_0))= 0$ only happens when
 $a=\left\lfloor \frac{t}{2}\right\rfloor=a+b$ and $s=\left\lceil\frac{t}{2}\right\rceil$, a contradiction to $A\not=A_0$.
We assume  $a> \left\lfloor \frac{t}{2}\right\rfloor$ for the remaining.
By (\ref{e6.3}) and (\ref{e6.4}), we have $\max(2a+b+1, s+a+b+1)\leq t+1$, so
$$\begin{array}{lll}
		a+s(a+b)&\leq &\left\{
                           \begin{array}{ll}
                             a(a+b+1), & \hbox{if $s<a$;} \\
                             s(a+b+1), & \hbox{if $s\geq a$}
                           \end{array}
                         \right. \medskip\\
               &\leq& \displaystyle\left\lfloor\frac{t+1}{2}\right\rfloor \left\lceil\frac{t+1}{2}\right\rceil\medskip\\
               &=&\displaystyle\left\{
                           \begin{array}{ll}
			\displaystyle\left\lfloor \frac{t}{2}\right\rfloor+\left\lfloor\frac{t}{2}\right\rfloor \left\lceil\frac{t}{2}\right\rceil+1	, & \hbox{if $t$ is odd;} \smallskip \\
                             \displaystyle\left\lfloor \frac{t}{2}\right\rfloor+\left\lfloor\frac{t}{2}\right\rfloor \left\lceil\frac{t}{2}\right\rceil, & \hbox{if $t$ is even.}
                           \end{array}
                         \right.
\end{array}$$
Putting this information to (\ref{e6.6}) and using  $\rho(A_0)<\rho(J_{c+1})=c+1$,
we have $f(\rho(A_0))>0$,
 except that $t$ is odd, $s=a=(t+1)/2$, and $b=-1.$
There is only one  such matrix
$$A=
\begin{pmatrix}
J_{s\times s} & J_{s\times (c-1-s)} & J_{s\times 1}& J_{s\times 1}\\
J_{(c-1-s)\times s} & J_{(c-1-s)\times (c-1-s)} & J_{(c-1-s)\times 1}& O_{(c-1-s)\times 1}\\
J_{1\times s} & J_{1\times (c-1-s)} &  0& 0\\
J_{1\times s} & O_{1\times (c-1-s)} &  0& 0\\
\end{pmatrix}\oplus O_{n-c-1},$$
where $s=(t+1)/2$. Thus $2\leq t=2s-1\leq 2c-3$ and $c\geq 3$.
To compute $\rho(A)=\rho(A[c+1|c+1])$,
observe that $\Pi_2:=\{\{1,\ldots,s\},\{s+1,\ldots,c-1\}, \{c\},\{c+1\}\}$ is an equitable  partition of $A[c+1|c+1]$ and the quotient matrix $\Pi_2(A[c+1|c+1])$ is
$$
\begin{pmatrix}
s &c-1-s & 1 & 1 \\
s   &c-1-s & 1 & 0 \\
s   &c-1-s & 0 & 0 \\
s   & 0    & 0 & 0
\end{pmatrix},
$$
which has characteristic polynomial
$$h(x)=x^4-(c-1)x^3+(1-c-s)x^2+s(c-1-s)x+s(c-1-s).$$
For $x\geq c$, the derivative of $h(x)$ satisfies
\begin{align*}
h'(x)=& 4x^3-3(c-1)x^2+2(1-c-s)x+s(c-1-s)\\
     \geq&x((4x-3(c-1))x+2(1-c-s))\\
     >& c((4c-3(c-1))c +2(1-c-(c-2)))\\
     \geq& c((c+3)c-4c+6)>0.
\end{align*}
Then $h(x)$ is strictly increasing in the interval $(c, \infty)$.
Using the polynomial $g$  in \eqref{e6.5},
$$h(x)-(x+1)g(x)=(c-1-s)x+c-2s.$$
If $t\leq 2c-5$, then
\begin{align*} h(\rho(A_0))=&(c-1-s)\rho(A_0)+c-2s\\
                           >&(c-1-(c-2))c+c-2(c-2)=4>0,
\end{align*}
concluding $\rho(A)<\rho(A_0)$, since $\rho(A_0)\in (c, \infty)$.
If $2\leq t=2c-3$, then $s=c-1$, $A=A'_0$ and $h(\rho(A_0))=c-2(c-1)=-c+2<0.$
Hence $\rho(A_0)<\rho(A'_0).$
\qed

\begin{rem}
The above proof also shows that in the case $t=2c-3$, $A_0$ is the unique graph in   $\mathscr{S}^*(n,e)$ attaining the second largest spectral radius.
\end{rem}

\subsection{$(0, 1)$-matrices with zero trace}\label{s6.2}

In this subsection, we study $(0, 1)$-matrices with zero trace.
This study is parallel to proving Conjecture C. For easier comparison, we suppress the meaning of $\mathscr{S}(n,e)$ in Conjecture~C and let $\mathscr{S}(n,e)$ denote the set of $n\times n$ $(0, 1)$-matrices with zero trace having exactly $e$ ones.
We will prove the following theorem.

\begin{thm}\label{thm6.3}
If  $e=c(c-1)+t$, where $2\leq t\leq 2c-1$ are integers, and $A\in \mathscr{S}(n,e)$ attains the maximum spectral radius, then there exists a permutation matrix $P$ such that $PAP^T$ or $PA^TP^T$ has the form
\begin{equation}\label{e6.7}
A_0=\begin{array}{ll}
    \begin{pmatrix}
        J_c-I_c &
            \begin{array}{c}
                J_{\bigl\lfloor\frac{t}{2}\bigr\rfloor \times 1}\\
                O_{(c-\bigl\lfloor\frac{t}{2}\bigr\rfloor)\times 1}
            \end{array} \\
        \begin{array}{cc}
            J_{1\times \bigl\lceil\frac{t}{2}\bigr\rceil} & O_{1\times
            (c-\bigl\lceil\frac{t}{2}\bigr\rceil)}
        \end{array} &0
    \end{pmatrix}\oplus O_{n-c-1}.
\end{array}
\end{equation}
\end{thm}

For the not mentioned cases $t=0$ and $t=1$ in Theorem~\ref{thm6.3},
Y. Jin and X. Zhang   \cite{jz:15} proved that $PAP^T=(J_c-I_c)\oplus O_{n-c}$ if $t=0$;
$PAP^T=((J_c-I_c)\oplus O_{n-c})+E_{ij}$ if $t=1$, where $i>c$ or $j>c$, and an additional situation  in $e=3$ $(c=2, t=1)$,
$$PAP^T=\begin{pmatrix}0 & 1 & 0\\ 0 & 0 &1 \\ 1& 0 & 0\end{pmatrix}\oplus O_{n-3}$$
is also possible. Y. Jin and X. Zhang  \cite{jz:15} also proved Theorem~\ref{thm6.3} in the cases that $t$ is relatively small.

Let $\mathscr{S}^*(n,e)$ denote the subset of $\mathscr{S}(n,e)$ which collects  matrices $A=(a_{ij})$ satisfying
$$\text{if $a_{ij}=1$, then $a_{hk}=1$ for $h\leq i$, $k\leq j$ and $h\ne k$}.$$
The following lemma is similar to Lemma~\ref{le6.1}.

\begin{lem}[\cite{jz:15}]\label{lem6.4}
If  $e=c(c-1)+t$, where $2\leq t\leq 2c-1$, and $A\in \mathscr{S}(n,e)$ attains the maximum spectral radius, then there exists a permutation matrix $P$ such that $PAP^T\in \mathscr{S}^*(n,e)$  and $PAP^T$ has the form
$A[k|k]\oplus O_{n-k}$ for some $k$ and the submatrix $A[k|k]$ is irreducible.
\end{lem}

\bigskip

\noindent{\it Proof of Theorem \ref{thm6.3}.}

 Let $e=c(c-1)+t$, where $2\leq t\leq 2c-1$. The matrix $A_0$ in (\ref{e6.7}) is in $\mathscr{S}^*(n,e)$.  We will show that $\rho(A)<\rho(A_0)$  for every $A\in \mathscr{S}^*(n,e)-\{A_0, A^T_0\}$.  Then Theorem~\ref{thm6.3} is proved by Lemma \ref{lem6.4}.
The proof is almost a copy of the proof of Conjecture~C. For the interested reader to check, we also provide the details. Other reader might go to the next section directly without missing important background.

By considering $A^T$ if necessary, we might assume that
the number of $1$'s in $A[c|c)$ is no larger than that in $A(c|c]$.
Let $(r_1, r_2, \ldots, r_n)$ denote the row-sum vector of $A$.  Since $A[c+1~|~c+1]$ is irreducible by Lemma~\ref{lem6.4}, $0<r_{c+1}<\max (c+1, t)$.  Let
$s:=r_{c+1}$.
Applying Corollary~\ref{cor5.1'} with $\ell=c+1,$ $C=A$, partition $\Pi=\{\{1\}$, $\{2\}$, $\ldots,$ $\{c\}$,  $\{c+1, c+2,\ldots, n\}\},$ and the following $(c+1)\times (c+1)$ rooted matrix
$$
M=\begin{pmatrix}
J_s-I_s & J_{s\times (c-s)}                 &\begin{array}{c} r_1-c+1\\ r_2-c+1 \\ \vdots   \end{array} \\
J_{(c-s)\times s}      & J_{c-s}-I_{c-s}  &\begin{array}{c}     \\  r_c-c+1\end{array}\\
J_{1\times s} &  O_{1\times (c-s)}               & 0
\end{pmatrix},
$$ we have $\rho(A)\leq \rho_r(M)$.
 To find $\rho_r(M)$, we consider the partition
$\Pi_1=\{\{1,2,\ldots,s\}$, $\{s+1, \ldots, c\}$, $\{c+1\}\}$ of $[c+1]$, and  observe that  $\Pi_1$ is an equitable partition of $M^T$. According to $s=c$ or $s<c$, the equitable quotient matrix $\Pi_1(M^T)$ has one of the following two forms
\begin{equation}\label{e6.8}
\begin{pmatrix}
c-1& 1\\
a & 0
\end{pmatrix}, \quad
\begin{pmatrix}
s-1 & c-s   & 1 \\
s   & c-s-1 & 0 \\
a   &  b    & 0
\end{pmatrix},
\end{equation}
respectively, where $a=\sum_{i=1}^s (r_i-c+1), \ b=\sum_{i=s+1}^c (r_i-c+1).$
Observe that the first matrix is the degenerated case of the second.
Note that $A=A_0$ implies $s=\lceil\frac{t}{2}\rceil$, $a=\lfloor\frac{t}{2}\rfloor$ and $b=0$. The converse is also true: Notice that $$a+b+s+\sum_{i=c+2}^n r_i=t$$
holds generally in $A$.
The conditions $s=\lceil\frac{t}{2}\rceil$, $a=\lfloor\frac{t}{2}\rfloor$ and $b=0$ imply
$A(c+1|c]=O_{(n-c-1)\times c}$ and $A[\{s+1, s+2, \ldots, c\}|c)=O_{(c-s)\times (n-c)}$, and $A(c+1|c]=O_{(n-c-1)\times c}$ also implies
$A[c|c+1)=O_{c\times (n-c-1)}$ from the shape of $A$ described in Lemma~\ref{lem6.4}, so $A=A_0$ follows.

We will provide important constraints between $s,a, b$ and $t$, which are exactly the same as in Section \ref{s6.1}. Let $r$ denote the number of off-diagonal zeros in $A[c|c].$ Then $a+b+r$ is  the number of $1$'s in $A[c|c)$. From the assumption in the beginning, we have  $a+b+r \leq (t+r)/2$.
Since the integer $a$ is at most the number of $1$'s in $A[s|c)$, we have
\begin{equation}\label{e6.9}
a\leq a+b+r\leq (t+r)/2.
\end{equation}
In particular, $a+b\leq (t-r)/2$ and
\begin{equation}\label{e6.10}
2a+b\leq t.
\end{equation}
Since the number of $1$'s in $A(c|s]$ is $t-a-b$, we have
\begin{equation}\label{e6.11}
 s\leq t-a-b.
\end{equation}
Since $\rho_r(M)=\rho_r(\Pi_1(M^T))$ by Lemma~\ref{le4.5}, it suffices to show $\rho(\Pi_1(M^T))< \rho(A_0)$.
The characteristic polynomial of $\Pi_1(M^T)$ in (\ref{e6.8}) is $f(x)/(x+1)$ or $f(x)$
according to $s=c$ or $s<c$, where
\begin{equation}\label{e.105}
f(x)=x^3-(c-2)x^2+(1-c-a)x+a(c-s-1)-sb.
\end{equation}
We use Calculus to study the shape of the polynomial $f(x)$.  If $x>c-1$, the derivative of $f(x)$ satisfies
\begin{align*}
f'(x)=&x(3x+2(2-c))+1-c-a\\
     >& (c-1)(3(c-1)+2(2-c))+1-c-a\\
     =& (c-1)c-a\geq c(c-1)-\frac{t+r}{2}\\
 \geq& c(c-1)-\frac{2c-1+(c-1)(c-2)}{2}\geq0
\end{align*}
by \eqref{e6.9}. Hence $f(x)$ is increasing in the interval $(c-1, \infty).$  Since $\rho(A_0)>\rho(K_c)=c-1$,
to prove $\rho_r(M)<\rho(A_0)$,
it suffices to show that $f(\rho(A_0))>0$.
Setting $s=\left\lceil\frac{t}{2}\right\rceil$, $a=\left\lfloor\frac{t}{2}\right\rfloor$, and $b=0$ in (\ref{e.105}), $\rho(A_0)$ is the largest zero of
\begin{equation}\label{e6.12}
g(x):=x^3-(c-2)x^2+\left(1-c-\left\lfloor\frac{t}{2}\right\rfloor\right)x+\left\lfloor\frac{t}{2}\right\rfloor\left(c-\left\lceil\frac{t}{2}\right\rceil-1\right).
\end{equation}
Hence
\begin{align}\label{e6.13}
f(\rho(A_0))=&f(\rho(A_0))-g(\rho(A_0)) \nonumber \\
            =&\left(\left\lfloor \frac{t}{2}\right\rfloor-a\right) (\rho(A_0)-c+1) -s\left(a+b\right)+ \left\lfloor\frac{t}{2}\right\rfloor \left\lceil\frac{t}{2}\right\rceil.
\end{align}

As in Section \ref{s6.1}, if $a\leq \left\lfloor \frac{t}{2}\right\rfloor$, then  we immediately have $f(\rho(A_0))\geq 0$ from (\ref{e6.13}), since $s+a+b\leq t$ in (\ref{e6.11}) implies $s(a+b)\leq \left\lfloor\frac{t}{2}\right\rfloor \left\lceil\frac{t}{2}\right\rceil$; and indeed $f(\rho(A_0))> 0$, since  $f(\rho(A_0))= 0$ only happens when
 $a=\left\lfloor \frac{t}{2}\right\rfloor=a+b$ and $s=\left\lceil\frac{t}{2}\right\rceil$, a contradiction to $A\not=A_0$.
We assume  $a> \left\lfloor \frac{t}{2}\right\rfloor$ for the remaining.
By (\ref{e6.10}) and (\ref{e6.11}), we have $\max(2a+b+1, s+a+b+1)\leq t+1$, so
$$\begin{array}{lll}
		a+s(a+b)&\leq &\left\{
                           \begin{array}{ll}
                             a(a+b+1), & \hbox{if $s<a$;} \\
                             s(a+b+1), & \hbox{if $s\geq a$}
                           \end{array}
                         \right. \medskip\\
               &\leq& \displaystyle\left\lfloor\frac{t+1}{2}\right\rfloor \left\lceil\frac{t+1}{2}\right\rceil\medskip\\
               &=&\displaystyle\left\{
                           \begin{array}{ll}
			\displaystyle\left\lfloor \frac{t}{2}\right\rfloor+\left\lfloor\frac{t}{2}\right\rfloor \left\lceil\frac{t}{2}\right\rceil+1	, & \hbox{if $t$ is odd;} \smallskip \\
                             \displaystyle\left\lfloor \frac{t}{2}\right\rfloor+\left\lfloor\frac{t}{2}\right\rfloor \left\lceil\frac{t}{2}\right\rceil, & \hbox{if $t$ is even.}
                           \end{array}
                         \right.
\end{array}$$
Putting this information to (\ref{e6.13}) and using  $\rho(A_0)<\rho(J_{c+1}-I_{c+1})=c$,
we have $f(\rho(A_0))>0$,
 except that $t$ is odd, $s=a=(t+1)/2$, and $b=-1.$
There are two  such matrices $A$ and $A^T$ and they have the same spectral radius, where
$$A=
\begin{pmatrix}
J_{s\times s}-I_s & J_{s\times (c-1-s)} & J_{s\times 1}& J_{s\times 1}\\
J_{(c-1-s)\times s} & J_{(c-1-s)\times (c-1-s)}-I_{c-1-s} & J_{(c-1-s)\times 1}& O_{(c-1-s)\times 1}\\
J_{1\times s} & J_{1\times (c-2-s)}~~~~0 &  0& 0\\
J_{1\times s} & O_{1\times (c-1-s)} &  0& 0\\
\end{pmatrix}\oplus O_{n-c-1}.$$
Note that $s=r_{c+1}\leq c-2$ in the above $A$, since a zero in the $(c,c-1)$ position causes a zero in $(c+1,c-1)$ position. To compute $\rho(A)=\rho(A[c+1|c+1])$,
observe that $\Pi_2:=\{\{1,\ldots,s\},\{s+1,\ldots,c-1\}, \{c\},\{c+1\}\}$ is an equitable  partition of $A[c+1|c+1]$ and the quotient matrix $\Pi_2(A[c+1|c+1])$ is
$$
\begin{pmatrix}
s-1 &c-1-s & 1 & 1 \\
s   &c-2-s & 1 & 0 \\
s   &c-2-s & 0 & 0 \\
s   & 0    & 0 & 0
\end{pmatrix},
$$
which has characteristic polynomial
$$h(x)=x^4-(c-3)x^3+(4-2c-s)x^2+((c-2)(s-1)-s^2)x-s(s-c+2).$$
Since the derivative of $h(x)$ is
\begin{align*} h'(x)=& 4x^3-3(c-3)x^2+2(4-2c-s)x+(c-2)(s-1)-s^2 \\
                    =& x(3x(x-(c-3))+8-4c-2s)+(x^3+(c-2)(s-1)-s^2),
\end{align*}
for $x\geq c-1\geq s+1$, we have $h'(x)\geq x(3(c-1)\cdot 2 +8-4c-2(c-2))>0$.
Then $h(x)$ is strictly increasing in the interval $(c-1, \infty)$.
Since $\rho(A_0)>\rho(J_c-I_c)=c-1$, to prove $\rho(A_0)>\rho(A)$, it suffices to show $h(\rho(A_0))>0$.
Since $\rho(A_0)$ is the zero of the polynomial $g$  in \eqref{e6.12}, we first compute  
$$h(x)-(x+1)g(x)=(c-1-s)x+c-1-2s,$$
and then find 
\begin{align*} h(\rho(A_0))=&(c-1-s)\rho(A_0)+c-1-2s\\
                           >&(c-1-(c-2))(c-1)+c-1-2(c-2)=2.
\end{align*}
\qed

\begin{rem}
If we follow the proof of Theorem \ref{thm6.3} to symmetric $(0,1)$-matrices with zero trace, the proof will be finished in the middle. Hence this gives an alternative proof of Conjecture B.
\end{rem}

\subsection{Nonnegative matrices with prescribed sum of entries}\label{s6.3}

We recall an old  result in 1987 due to R. Stanley  \cite{s:87}.

\begin{thm}[\cite{s:87}]\label{th6.6}
Let $C=(c_{ij})$ be an $n\times n$ symmetric $(0,1)$-matrix with zero trace. Let $2e$ be the number of 1's of $C$. Then
$$\rho(C)\leq \frac{-1+\sqrt{1+8e}}{2},$$
with equality if and only if
$$e=\frac{k(k-1)}{2}$$
and there exists a permutation matrix $P$ such that
$$
PCP^T=(J_k-I_k)\oplus O_{n-k}
$$
for some positive integer $k$.
\end{thm}

The following theorem generalizes Theorem~\ref{th6.6} to nonnegative matrices, not necessarily symmetric.
\begin{thm}
Let $C=(c_{ij})$ be an $n\times n$ nonnegative matrix. Let $m$ be the sum of entries
and $d$ (resp. $f$) be the largest diagonal element (resp. the largest off-diagonal element) of $M$.
Then
 \begin{equation}\label{e6.14}
\rho(C)\leq \frac{d-f+\sqrt{(d-f)^2+4mf}}{2}.
\end{equation}
Moreover, if $f>0$ then the equality in (\ref{e6.14}) holds
 if and only if
there exists a permutation matrix $P$ such that
\begin{equation}\label{e6.14.5}PCP^T=(fJ_k+(d-f)I_k)\oplus O_{n-k}\end{equation}
for some positive integer $k$.
\end{thm}

\begin{proof} If $f=0$ then the nonzero entries only appear in the diagonal of $C$, so $\rho(C)\leq d$ and
(\ref{e6.14}) holds. Assume $f>0$.
Consider the $(n+1)\times (n+1)$ nonnegative matrix $C^+=C\oplus O_1$ which  has row-sum vector $(r_1,r_2,\ldots,r_n,r_{n+1})$ with $r_{n+1}=0$ and  a nonnegative left eigenvector $v^T$ for $\rho(C^+)=\rho(C)$.
Let $C'=M_{n+1}$ defined in (\ref{e4.2}) with $f_1=f_2=f$ such that
 $C'$ has the same row-sum vector as $C^+$ and $C^+[n|n]\leq C'[n|n]=fJ_n+(d-f)I_n$.
 Note that $C'$ has a positive rooted eigenvector $v'=(v'_1, v'_2, \ldots, v'_{n+1})^T$ for $\rho_r(C')$ by Lemma~\ref{le4.6}, so $v^Tv'>0.$ Hence the assumptions of Theorem~\ref{th3.4}
hold with $(C, \lambda, \lambda')=(C^+, \rho(C^+), \rho_r(C'))$.
 Now by Theorem~\ref{th3.4} and Lemma~\ref{le4.7}(ii), we have
$$\rho(C)=\rho(C^+)\leq \rho_r(C')=\frac{d-f+\sqrt{(d-f)^2+4mf}}{2},$$
finishing the proof of the first part.

To prove the sufficiency in the  second part, assume $f>0$ and (\ref{e6.14.5}). Note that $m=k^2f+k(d-f)$. Using $\rho(C)=\rho(PCP^T)=\rho(fJ_k+(d-f)I_k),$ we have
$$\rho(C)=kf+(d-f)=\frac{d-f+\sqrt{(d-f)^2+4mf}}{2}.$$

For proving necessity, assume $f>0$ and $\rho(C)=\rho_r(C')$. Then $C\ne O_n$ and $C^+\ne O_{n+1}$.
We will apply Theorem \ref{th3.4} with $C=C^+=(c^+_{ij})$. Let $v^T=(v_1, v_2, \ldots, v_{n+1})$ be a nonnegative left eigenvector of $C^+$ for $\rho(C^+)$. Then $v_{n+1}=0$.
By rearranging the indices of $C$ we might assume that for some $k\leq n$,  $r_i>0$ for $1\leq i\leq k$ and  $C^+(k|n+1]=O_{(n+1-k)\times (n+1)}$. Since $r_i$ are also row-sums of $C'$ and by Lemma~\ref{le4.3}(ii), we have $v'_j>v'_{n+1}$ for $1\leq j\leq k$.
By Theorem \ref{th3.4}(b), we have  $c^+_{ij}=c'_{ij}$ for $1\leq i\leq n,$ $1\leq j\leq k$ with $v_i\not=0$. With  $v[k]^T:=(v_1, v_2, \ldots, v_k),$ we have
\begin{equation}\label{e6.15} \rho(C)v^T=v^TC^+=v[k]^TC^+[k|n+1]\end{equation}
and
\begin{equation}\label{e6.16}
v[k]^TC^+[k|k]=v[k]^TC'[k|k]=v[k]^T(fJ_k+(d-f)I_k).
\end{equation}
From (\ref{e6.15}) and  (\ref{e6.16}), we have
$$\rho(C)v[k]^T=v[k]^TC^+[k|k]=v[k]^T(fJ_k+(d-f)I_k).$$
Since $v[k]^T$ is a left eigenvector of the irreducible matrix $fJ_k+(d-f)I_k$, $v[k]^T$ is positive, 
so $C[k|k]=C^+[k|k]=C'[k|k]=fI_k+(d-f)I_k$.
For $k+1\leq i\leq n+1$, we have $v_i=0$ by Theorem \ref{th3.4}(b), since $0=c^+_{i1}\ne c'_{i1}=f$.
Using (\ref{e6.15}) again, the last $n+1-k$ columns of $C^+$ are zero. Hence $C^+=(fJ_k+(d-f)I_k)\oplus O_{n+1-k}$, and consequently $C=(fJ_k+(d-f)I_k)\oplus O_{n-k}$.
\end{proof}

\section{Concluding remarks}

In this paper, we have developed a tool to obtain both upper and lower bounds for the spectral radius of a nonnegative matrix. This assists obtaining known and new results more easily and also make it easier to improve these results. The key tool, Theorem \ref{th3.4}, and the main result, Corollary~\ref{cor5.1'}, are special cases of Lemma~\ref{le3.1} when a specific triple $(C', P, Q)$ of square matrices is chosen according to a given square matrix $C$. Let us consider the simplest case that $C$ is a binary matrix. By choosing $P=I_n$, $Q=I_n+\sum_{i=1}^n E_{in}$, and $C'$  obtained from $J_n$ by changing the entries in its last column to maintain the same row-sums with $C$, we can
check easily at least the condition $PCQ\leq PC'Q$ in Lemma~\ref{le3.1}(i) holds
and expect the remaining conditions (ii)-(iv) hold to conclude that $\rho(C)\leq \rho(C').$

We believe it is worth focusing more on investigating the bounds derived from selecting other triples of $(C', P, Q)$, as this exploration may be helpful in solving many extremal problems related to graphs and their spectral radii. For example, a result of P. Csikv\'{a}ri in 2009 \cite{c:09} stating that the spectral radius of a symmetric $(0, 1)$ matrix $C$ will not be decreased after a {\it Kelmans transformation} \cite{k:81} can be reproved by the method in this paper taking $P=I_n+E_{ij}=Q^T$ and $C'$ to be the Kelmans transformation of $C$ from $i$ to $j$. This result extends to a nonsymmetric matrix with a minor assumption by essentially the same proof \cite{kw:21}.

\section*{Acknowledgements}
This research is supported by the Ministry of Science and Technology of Taiwan
R.O.C. under the projects MOST 110-2811-M-A49-505 and MOST 111-2115-M-A49-005-MY2.

\end{document}